\documentclass[12pt,english]{article}
\usepackage[T1]{fontenc}
\usepackage[latin9]{inputenc}
\usepackage{geometry}
\usepackage{babel}
\usepackage{enumitem}
\usepackage{amsmath}
\usepackage{amsthm}
\usepackage{amssymb}
\usepackage{microtype}
\usepackage[pdftex,
 bookmarks=true,bookmarksnumbered=false,bookmarksopen=false,
 breaklinks=false,pdfborder={0 0 1},backref=false,colorlinks=false]
 {hyperref}
\hypersetup{pdftitle={Product of commutators in a division ring of skew Laurent series over a field},
 pdfkeywords={skew Laurent polynomial ring, product of commutators, }}

\makeatletter
\newcommand{\address}[1]{
\vspace*{-3pc}
\begin{center} #1 \end{center}
}
\theoremstyle{plain}
\newtheorem*{thm*}{\protect\theoremname}
\theoremstyle{plain}
\newtheorem{thm}{\protect\theoremname}
\theoremstyle{plain}
\newtheorem{lem}[thm]{\protect\lemmaname}

\date{}

\makeatother

\providecommand{\lemmaname}{Lemma}
\providecommand{\theoremname}{Theorem}

\begin{document}
\title{Commutator Products in Skew Laurent Series Division Rings}
\author{Hau-Yuan Jang\thanks{Partially supported by National Science and Technology Council, Taiwan,
grant \#113-2115-M-006-009.} \thanks{Email: l18121022@gs.ncku.edu.tw}\ ~and\ Wen-Fong Ke$^{*}$\thanks{Email: wfke@mail.ncku.edu.tw}}
\maketitle

\address{
Department of Mathematics, National Cheng Kung University, Tainan
70101, Taiwan}
\vspace{1pc}

\begin{abstract}
In 1965, Baxter established that a simple ring is either a field or
that every one of its elements can be expressed as a sum of products
of commutator pairs. In a recent paper, Gardella and Thiel demonstrated
that every element in a noncommutative division ring can be represented
as the sum of just two products of two commutators. They further posed
the question of whether every element in a noncommutative division
ring can be represented as the product of two commutators. In this
paper, we affirmatively answer this question for skew Laurent series
division rings over fields.
\end{abstract}
{\small\medskip{}
}\\
{\small 2020 Mathematics Subject Classification: 12E15, 16K20}{\small\par}

\section{Introduction}

\global\long\def\kbx{k(\!(x)\!)}%
\global\long\def\kbxs{k(\!(x^{s})\!)}%
\global\long\def\kbxn{k(\!(x^{n})\!)}%
\global\long\def\kdbx{k(\!(\sigma;x)\!)}%
\global\long\def\kzbxs{k_{0}(\!(x^{s})\!)}%
\global\long\def\kzbxn{k_{0}(\!(x^{n})\!)}%
For elements $a$ and $b$ in a ring, their commutator is defined
as $[a,b]=ab-ba$. A fundamental result by Baxter \cite[Theorem 1]{Baxter}
establishes that a simple ring is either a field or that every one
of its elements can be expressed as a sum of products of commutator
pairs.

Most recently, Gardella and Thiel \cite[Proposition 5.8]{GardellaT}
strengthened this result for division rings, demonstrating that every
element in a noncommutative division ring can be represented as the
sum of just two products of two commutators. They then posed the question
\cite[Question 5.9]{GardellaT} of whether every element in a noncommutative
division ring can be represented as the product of two commutators.

In this paper, we provide an affirmative answer to this question for
skew Laurent series division rings over fields. Specifically, we prove
the following statement:
\begin{thm*}
Let $D$ be a skew Laurent series division ring over a field. Then
every element of $D$ can be expressed as a product of two commutators.
\end{thm*}
For our discussions, we fix a field $k$ and let $\sigma$ be an automorphism
of $k$ that is not the identity map.

Laurent series over $k$ are expressions of the form $\sum_{i=s}^{\infty}a_{i}x^{i}$,
where $x$ is an indeterminate, $s\in\mathbb{Z}$ and $a_{i}\in k$
for all $i\geq s$. The set of all Laurent series over $k$, denoted
by $\kbx$, forms a field under the usual addition and multiplication.

The skew Laurent series division ring $\kdbx$ consists of the Laurent
series over $k$, but with the additional multiplication rule $x^{i}a=\sigma^{i}(a)x$
for all $a\in k$ and $i\in\mathbb{Z}$. Here, $\sigma^{i}$ denotes
the $i$-th iterate of the automorphism $\sigma$. For further details
on such skew constructions, see \cite[Examples (1.7) and (1.8)]{Lam}.

The center $Z(D)$ of $D$ is
\begin{itemize}
\item $k_{0}=\{a\in k\mid\sigma(a)=a\}$ when $\sigma$ is of infinite order,
and
\item $\kzbxn$ if $\sigma$ is of finite order $n$. (See \cite[Proposition 14.2]{Lam}.)
\end{itemize}
\par\noindent In the later case, if we put $F=\kzbxn$ and $K=\kbxn$,
then $F$ and $K$ are subfields of $D$ such that $D$ is a centrally
simple cyclic division algebra over $F$ of degree~$n$ and $[D:K]=n$.
Moreover, $K$ is a maximal subfield of $D$ and $D\otimes_{F}K\cong M_{n}(K)$.
(See \cite[Theorem 14.6 and (14.13)]{Lam}.) For an element $f\in D$,
the trace of the matrix in $M_{n}(K)$ corresponding to $f\otimes1$
is called the reduced trace of $f$.

In the next section, we first prove the theorem for the cases where
the order of $\sigma$ is $2$ or $3$, utilizing the fact that $D$
is a centrally simple cyclic division algebra over $F$. Subsequently,
after establishing some technical lemmas, we address the case where
the order of $\sigma$ is at least $5$. Finally, we consider the
case where $\sigma$ has order~$4$ in the last subsection.

\section{Proof of the theorem}

\subsection{When $\sigma$ is of order $2$ or $3$.}

Let $\sigma$ be of order $n$, $n\in\{2,3\}$. Thus, $D$ is a centrally
simple cyclic division algebra over $F$ of degree $n=2$ or $3$,
accordingly. In such cases, by Corollary 0.9 and Theorem 0.10 in \cite{AmitsurR},
an element $f\in D$ is a commutator if and only if the reduced trace
of $f$ is zero.

Let $f$ be a nonzero element of $D$, and let $M$ be a maximal subfield
of $D$ containing $f$. By the proposition in \cite{Haile}, there
is a nonzero element $g\in D$ such that the reduced trace of $gh$
is zero for all $h\in M$. Moreover, for this $g$, there is an $F$-subspace
$W$ of $M$ with $\dim_{F}W=n-1\geq1$ such that the reduced trace
of $w^{-1}g$ and that of $g^{-1}w$ are zero for all $w\in W\setminus\{0\}$.
Now, we fix a nonzero element $w_{0}\in W$. Then $g^{-1}w_{0}$ is
of reduced trace zero. Also, from $fw_{0}^{-1}\in M$, we have that
$fw_{0}^{-1}g$ is of reduced trace zero as well. By the above mentioned
results of Amitsur and Rowen, both $fw_{0}^{-1}g$ and $g^{-1}w_{0}$
are commutators. Hence, $f=(fw_{0}^{-1}g)(g^{-1}w_{0})$ is a product
of two commutators. This finishes the case when $\sigma$ is of order
$2$ or of order~$3$.

\subsection{Some observations.}

For an element $a\in k$, we said that $a$ is of $\sigma$-degree
$m$ and denote it by $\deg_{\sigma}(a)=m$, where $m\in\mathbb{N}$,
if $\sigma^{j}(a)\neq a$ for $j\in\{1,...,m-1\}$ and $\sigma^{m}(a)=a$.
If $\sigma^{j}(a)\neq a$ for all $j>0$, then we say that $a$ is
of infinite $\sigma$-degree and write $\deg_{\sigma}(a)=\infty$.

Let $b\in k\setminus\{0\}$ be fixed. We notice that if $i\in\mathbb{Z}$
and $\sigma^{i}(b)\not=b$, then, for any $a\in k$,
\begin{align}
[b,(b-\sigma^{i}(b))^{-1}ax^{i}] & =b(b-\sigma^{i}(b))^{-1}ax^{i}-(b-\sigma^{i}(b))^{-1}ax^{i}b\nonumber \\
 & =b(b-\sigma^{i}(b))^{-1}ax^{i}-(b-\sigma^{i}(b))^{-1}\sigma^{i}(b)ax^{i}\nonumber \\
 & =(b(b-\sigma^{i}(b))^{-1}-\sigma^{i}(b)(b-\sigma^{i}(b))^{-1})ax^{i}\nonumber \\
 & =(b-\sigma^{i}(b))(b-\sigma^{i}(b))^{-1}ax^{i}\nonumber \\
 & =ax^{i}.\label{eq:(1)}
\end{align}

There are two cases to be considered.

\paragraph*{Case 1. $\boldsymbol{\deg_{\sigma}(b)=\infty}$. }

Let $\sum_{i=s}^{\infty}a_{i}x^{i}\in D$. Put $n=-|s|-1\leq-1$ and
set $m=s-n\geq1$. Then
\begin{align*}
\sum_{i=s}^{\infty}a_{i}x^{i}=\sum_{j=m}^{\infty}a_{n+j}x^{n+j} & =\sum_{j=m}^{\infty}\sigma^{n}(\sigma^{-n}(a_{n+j}))x^{n+j}\\
 & =\sum_{j=m}^{\infty}\sigma^{n}(b_{j})x^{n}\cdot x^{j}=\sum_{j=m}^{\infty}x^{n}\cdot(b_{j}x^{j})=x^{n}\cdot\sum_{j=m}^{\infty}b_{j}x^{j},
\end{align*}
where $b_{j}=\sigma^{-n}(a_{n+j})$ for $j=m,m+1,\dots$. From (\ref{eq:(1)}),
we have
\[
x^{n}=[b,(b-\sigma^{n}(b))^{-1}x^{n}]
\]
and
\[
\sum_{j=m}^{\infty}b_{j}x^{j}=\sum_{j=m}^{\infty}[b,(b-\sigma^{j}(b))^{-1}b_{j}x^{j}]=\Big[b,\sum_{j=m}^{\infty}(b-\sigma^{j}(b))^{-1}b_{j}x^{j}\Big].
\]
We have proved the following.
\begin{lem}
\label{lem:infinite}If $\deg_{\sigma}b=\infty$, then every element
of $D$ is a product of two commutators.
\end{lem}

\paragraph*{Case 2. $\boldsymbol{\deg_{\sigma}b=n<\infty}$.}

In this case, we have $\sigma^{0}(b)=b$, and
\[
\sigma^{qn}(b)=\underbrace{\sigma^{n}(\dots(\sigma^{n}(\sigma^{n}}_{|q|\text{-times}}(b))\dots)=b
\]
for all $q\in\mathbb{Z}\setminus\{0\}$. If $m=qn+r$ where $q,r\in\mathbb{Z}$
and $0<r<n$, then $\sigma^{m}(b)=\sigma^{r}(\sigma^{qn}(b))=\sigma^{r}(b)\not=b$.
Thus, $\sigma^{m}(b)\not=b$ if $n\nmid m$.

Let $\sum_{i=s}^{\infty}a_{i}x^{i}\in D$ be such that $a_{i}=0$
whenever $n\mid i$; thus, $\sum_{i=s}^{\infty}a_{i}x^{i}=\sum_{i\geq s;n\nmid i}a_{i}x^{i}$.
Using (\ref{eq:(1)}), we have
\begin{align}
\sum_{i\geq s;n\nmid i}a_{i}x^{i} & =\sum_{i\geq s;n\nmid i}[b,(b-\sigma^{j}(b))^{-1}a_{i}x^{i}]\nonumber \\
 & =\Big[b,\sum_{i\geq s;n\nmid i}(b-\sigma^{j}(b))^{-1}a_{i}x^{i}\Big]\in[b,D_{n}]\subseteq[b,D].\label{eq:(2)}
\end{align}
 Let $D_{n}$ be the collection of all such elements of $D$.
\begin{lem}
\label{lem:i+j,i=000020neq=000020j}Let $f=\sum_{i=s}^{\infty}a_{i}x^{i}\in D$
with $a_{s}\not=0$. Suppose that there are $u,v\in\mathbb{Z}$ such
that $n\nmid(u-v)$, $n\nmid u$, $n\nmid v$ and $s=u+v$. Then $f=gh$
for some $g,h\in D_{n}$. Consequently, $f$ is a product of two commutators.
\end{lem}

\begin{proof}
For $g,h\in D$, $g=\sum_{j=u}^{\infty}b_{j}x^{j}$ and $h=\sum_{j'=v}^{\infty}c_{j'}x^{j'}$,
the following equations need to be satisfied in order to get $f=gh$:
\begin{align*}
 & b_{u}\sigma^{u}(c_{v})=a_{s},\\
 & b_{u}\sigma^{u}(c_{v+1})+b_{u+1}\sigma^{u+1}(c_{v})=a_{s+1},\\
 & b_{u}\sigma^{u}(c_{v+2})+b_{u+1}\sigma^{u+1}(c_{v+1})+b_{u+2}\sigma^{u+2}(c_{v})=a_{s+2},\\
 & \quad\vdots\\
 & b_{u}\sigma^{u}(c_{v+t})+b_{u+1}\sigma^{u+1}(c_{v+t-1})+\cdots+b_{u+t}\sigma^{u+t}(c_{v})=a_{s+t},\\
 & b_{u}\sigma^{u}(c_{v+t+1})+b_{u+1}\sigma^{u+1}(c_{v+t})+\cdots+b_{u+t+1}\sigma^{u+t+1}(c_{v})=a_{s+t+1},\\
 & \quad\vdots
\end{align*}

For each $t\geq0$, we shall find coefficients $b_{u},b_{u+1},...,b_{u+t},c_{v},c_{v+1},...,c_{v+t}\in k$
such that
\begin{enumerate}[label=(E.\arabic{enumi})]
\item the first $t+1$ of these equations hold, and
\item $b_{j}=0$ whenever $j\in n\mathbb{Z}$ and $c_{j'}=0$ whenever $j'\in n\mathbb{Z}$.
\end{enumerate}
For $t=0$, set $b_{u}=a_{s}$ and $c_{v}=1$, and (E.1) is satisfied.
Since $u\notin n\mathbb{Z}$ and $v\notin n\mathbb{Z}$ by the assumptions,
(E.2) is satisfied as well.

For $t=1$, either $u+1\in n\mathbb{Z}$ or $u+1\not\in n\mathbb{Z}$.
If $u+1\notin n\mathbb{Z}$, we take $c_{v+1}=0$ and
\[
b_{u+1}=\sigma^{u+1}(c_{v})^{-1}(a_{s+1}-b_{u}\sigma^{u}(c_{v+1}))=a_{s+1};
\]
if $u+1\in n\mathbb{Z}$, then $v+1\not\in n\mathbb{Z}$ by the assumptions,
and we take $b_{u+1}=0$ and
\[
c_{v+1}=\sigma^{-u}(b_{u}^{-1}(a_{s+1}-b_{u+1}\sigma^{u+1}(c_{v})))=\sigma^{-u}(a_{s}^{-1}a_{s+1}).
\]
In both cases, (E.1) and (E.2) are satisfied.

Assuming that $t\geq1$, and that we have found $b_{u},b_{u+1},\dots,b_{u+t},c_{v},c_{v+1},\dots,c_{v+t}$
such that (E.1) and (E.2) are satisfied. Again, we have to consider
whether $u+t+1$ is in $n\mathbb{Z}$ or not. If $u+t+1\not\in n\mathbb{Z}$,
we take $c_{v+t+1}=0$ (note that $v+t+1$ may or may not be in $n\mathbb{Z}$)
and
\[
b_{u+t+1}=\sigma^{u+t+1}(c_{v})^{-1}(a_{s+t+1}-(b_{u+1}\sigma^{u+1}(c_{v+t})+\cdots+b_{u+t}\sigma^{u+t}(c_{v+1}))).
\]
If $u+t+1\in n\mathbb{Z}$, then $v+t+1\notin n\mathbb{Z}$; we take
$b_{u+t+1}=0$ and
\[
c_{v+t+1}=\sigma^{-u}(b_{u}^{-1}(a_{s+t+1}-(b_{u+1}\sigma^{u+1}(c_{v+t})+\cdots+b_{u+t}\sigma^{u+t}(c_{v})))).
\]
In both cases, (E.1) and (E.2) are satisfied.

Continuing the process, we eventually find $g,h\in D_{n}$ such that
$f=gh$ as claimed. The last statement follows from (\ref{eq:(2)}).

Using Lemma \ref{lem:i+j,i=000020neq=000020j}, we can deduce the
following result.
\end{proof}
\begin{lem}
\label{lem:n=000020geq=0000205}If $\deg_{\sigma}b\geq5$, then every
element of $D$ is a product of two commutators.
\end{lem}

\begin{proof}
Suppose that $\deg_{\sigma}b=n\geq5$. If we can show that every $s\in\mathbb{Z}$
can be written as a sum $s=u+v$ for some $u,v\in\mathbb{Z}$ such
that $n\nmid(u-v)$, $n\nmid u$ and $n\nmid v$, then Lemma \ref{lem:i+j,i=000020neq=000020j}
applies, and we are done. Equivalently, if we can show that the set
$S=\{u+v\mid u,v\in\mathbb{Z}_{n},u\neq0,v\neq0,u\neq v\}$ is $\mathbb{Z}_{n}$,
then we are done. Let $m\in\mathbb{N}$ with $3\leq m\leq n$. Then
$m=1+(m-1)$ and $m-1\not=1$. Therefore, $\mathbb{Z}_{n}\setminus\{\overline{1},\overline{2}\}\subseteq S$.
With $n\geq5$, we have $2\neq n-1$ and $\overline{1}=\overline{2}+\overline{n-1}$;
also, $3\neq n-1$ and $\overline{2}=\overline{3}+\overline{n-1}$.
Therefore, $S=\mathbb{Z}_{n}$, and the result holds.
\end{proof}

\subsection{When $\sigma$ is of order at least $5$.}

Now, we assume that the order of $\sigma$ is at least $5$.

First, we consider the case when $\sigma$ is of infinite order. If
there exists an element $b\in k$ such that $\sigma^{i}(b)\neq b$
for all $i\in\mathbb{N}$, then Lemma \ref{lem:infinite} ensures
that every element of $D$ is a product of two commutators. Suppose
that every element in $k$ is of finite $\sigma$-degree, and let
$S=\{\deg_{\sigma}(a)\mid a\in k\setminus\{0\}\}$. We claim that
$S$ is infinite. For, if $S$ is finite, and $m$ is the least common
multiple of $S$, then $\sigma^{m}(a)=a$ for all $a\in k\setminus\{0\}$.
Certainly $\sigma^{m}(0)=0$. This contradicts to the assumption that
$\sigma$ is of infinite order. Therefore, $S$ is infinite, and so
there is some $b\in k$ with $\deg_{\sigma}(b)\geq5$. Lemma~\ref{lem:n=000020geq=0000205}
applies, and we have that every element of $D$ is a product of two
commutators.

Next, consider the case when $\sigma$ is of finite order $n$, $n\geq5$.
As $k_{0}$ is the fixed field of $\sigma$, the Galois group $G$
of the field extension $k/k_{0}$ is the cyclic group $G=\langle\sigma\rangle$,
and there is an element $y\in k$ such that $\{y,\sigma(y),\dots,\sigma^{n-1}(y)\}$
forms a $k_{0}$-basis of $k$ by the normal basis theorem. In particular,
$\deg_{\sigma}y=n\geq5$. Again, from Lemma \ref{lem:n=000020geq=0000205},
we have that every element of $D$ is a product of two commutators.

\subsection{When $\sigma$ is of order $4$.}

We now assume that $\sigma$ is of order $4$. Then $\dim_{k_{0}}k=\vert\langle\sigma\rangle\vert=4$
and there is some $y\in k$ such that $\{y,\sigma(y),\sigma^{2}(y),\sigma^{3}(y)\}$
is a basis of $k$ over $k_{0}$. For the discussions in this subsection,
all the involved subspaces of $k$ and vectors are over the field
$k_{0}$, and ``independent'' means ``linearly independent over
$k_{0}$''.

For $a\in k$, we have $[x,ax^{i}]=xax^{i}-ax^{i+1}=(\sigma(a)-a)x^{i+1}$.
Therefore, $[x,D]=\{\sum_{i=s}^{\infty}a_{i}x^{i}\mid a_{i}\in L\}$,
where $L=\text{Im}(\sigma-1)=\{\sigma(z)-z\mid z\in k\}$, a subspace
of $k$ having $\{y-\sigma(y),\sigma(y)-\sigma^{2}(y),\sigma^{2}(y)-\sigma^{3}(y)\}$
as a basis.
\begin{lem}
\label{lem:l.indep.}If $a,b\in k$ are independent, then $aL+bL=k$.
\end{lem}

\begin{proof}
Since $\dim_{k_{0}}L=3$ and $\dim_{k_{0}}k=4$, it suffices to show
that $aL\not=bL$. Assume on the contrary that $aL=bL$. Equivalently,
$a^{-1}bL=L$. Let $y_{0}\in k\setminus L$ and consider the element
$a^{-1}by_{0}\in k$. One can find $\alpha\in k_{0}$ and $y'\in L$
such that $a^{-1}by_{0}=\alpha y_{0}+y'$. Thus $(a^{-1}b-\alpha)y_{0}=y'\in L$.
If $z\in k$, then $z=\beta y_{0}+y''$ for some $\beta\in k_{0}$
and $y''\in L$. From $\beta(a^{-1}b-\alpha)y_{0}\in\beta L\subseteq L$,
$a^{-1}by''\in a^{-1}bL=L$ and $y''\in L$, we have
\[
(a^{-1}b-\alpha)z=\beta(a^{-1}b-\alpha)y_{0}+a^{-1}by''-\alpha y''\in L.
\]
Thus, $\dim_{k_{0}}(a^{-1}b-\alpha)k\leq\dim_{k_{0}}L=3<\dim_{k_{0}}k=4$.
This is only possible when $a^{-1}b-\alpha=0$, or equivalently, when
$b=\alpha a$, which contradicts the fact that $a$ and $b$ are independent.
Therefore, $aL\not=bL$. From this, we get $\dim_{k_{0}}(aL+bL)>\dim_{k_{0}}aL=3$,
and so $aL+bL=k$.
\end{proof}
We need two subspaces of $L$. Let
\[
k_{1}=\{z\in k\mid\sigma(z)=-z\}\;\text{ and }\;k_{2}=\{z\in k\mid\sigma^{2}(z)=-z\}.
\]
 Clearly, $k_{1}$ and $k_{2}$ are $\sigma$-invariant subspaces
of $k$. Let us investigate $k_{1}$ and $k_{2}$ further.

For $a=\sum_{j=0}^{3}a_{j}\sigma^{j}(y)\in k_{1}$, where $a_{j}\in k_{0}$,
we have
\[
-\sum_{j=0}^{3}a_{j}\sigma^{j}(y)=-a=\sigma(a)=\sum_{j=0}^{3}a_{j}\sigma^{j+1}(y),
\]
and so $a_{j+1}=-a_{j}$ for $j\in\{0,1,2\}$ and $a_{3}=-a_{0}$.
Hence $a=a_{0}y-a_{0}\sigma(y)+a_{0}\sigma^{2}(y)-a_{0}\sigma^{3}(y)=a_{0}e_{1}$,
where $e_{1}=y-\sigma(y)+\sigma^{2}(y)-\sigma^{3}(y)=\sigma(\sigma(y)-\sigma^{2}(y))-(\sigma(y)-y)\in L$.
Therefore, $k_{1}=\text{span}_{k_{0}}\{e_{1}\}$, and is a subspace
of $L$.

For $b=b_{0}y+b_{1}\sigma(y)+b_{2}\sigma^{2}(y)+b_{3}\sigma^{3}(y)\in k_{2}$,
where $b_{0},b_{1},b_{2},b_{3}\in k_{0}$, we have
\[
-\sum_{j=0}^{3}b_{j}\sigma^{j}(y)=-b=\sigma^{2}(b)=b_{0}\sigma^{2}(y)+b_{1}\sigma^{3}(y)+b_{2}y+b_{3}\sigma(y),
\]
and so $b_{0}=-b_{2}$ and $b_{1}=-b_{3}$. As a result, $b=b_{0}e_{2}+b_{1}\sigma(e_{2})$,
where $e_{2}=y-\sigma^{2}(y)\in k_{2}$. It is clear that $e_{2}$
and $\sigma(e_{2})$ are independent, and $k_{2}=\text{span}\{e_{2},\sigma(e_{2})\}$.
Note that $e_{2}=y-\sigma^{2}(y)=(\sigma-1)\left(-y-\sigma(y)\right)\in L$,
and $\sigma(e_{2})=\sigma(y)-\sigma^{3}(y)=(\sigma^{2}-1)(-\sigma(y))=(\sigma-1)(-\sigma^{2}(y)-\sigma(y))\in L$.
We have $k_{2}=\text{span}\{e_{2},\sigma(e_{2})\}$, and is a subspace
of $L$ as well.

We also note that if $z\in k_{2}\setminus\{0\}$, then $1=\sigma^{2}(zz^{-1})=\sigma^{2}(z)\sigma^{2}(z^{-1})=-z\sigma^{2}(z^{-1})$,
and so $\sigma^{2}(z^{-1})=-z^{-1}$. This shows that $z^{-1}\in k_{2}$.
We shall use this fact to prove the next result.
\begin{lem}
\label{lem:special_conditions}Let $c\in k$ be such that $\sigma^{2}(c)\not=c$
or $\sigma^{2}(c)=c$ but $c\not\in k_{1}$, then there exist $a,b\in L$
such that $c=ab$ and $aL+\sigma^{i}(b)L=k$ for all $i\in\mathbb{Z}$.
\end{lem}

\begin{proof}
Assume first that $\sigma^{2}(c)\not=c$. Set $W=\{cz\mid z\in k_{2}\}$.
Then $W$ is a subspace of $k$ with $\dim_{k_{0}}W=\dim_{k_{0}}k_{2}=2$.
From $\dim_{k_{0}}W+\dim_{k_{0}}L=5>4=\dim_{k_{0}}k$, we have $W\cap L\not=\{0\}$.
Suppose that $ca'=b'$, where $a'\in k_{2}$ and $b'\in L$. Write
$a'=a^{-1}$, where $a\in k_{2}$, we get $c=ab'$. Thus, $\sigma^{2}(c)=\sigma^{2}(ab')=-a\sigma^{2}(b')$.
Since $\sigma^{2}(c)\not=c=ab'$, we see that $\sigma^{2}(b')\not=-b'$.
Consequently, $\sigma^{2}(b')\not\in k_{2}$, and so $a$ and $\sigma^{2}(b')$
are independent. Let $b=\sigma^{2}(b')$. We have $aL+bL=k$ by Lemma
\ref{lem:l.indep.}. Notice that $\sigma^{i}(b)\not\in k_{2}$ for
all $i\in\mathbb{Z}$ as $\sigma(k_{2})=k_{2}$ and $b\not\in k_{2}$.
Thus, for every $i\in\mathbb{Z}$, $a$ and $\sigma^{i}(b)$ are independent
as well; hence $aL+\sigma^{i}(b)L=k$ by Lemma \ref{lem:l.indep.}.

Next, we assume that $\sigma^{2}(c)=c$ but $c\not\in k_{1}$. If
both $\{e_{2},e_{2}^{-1}c\}$ and $\{e_{2},\sigma(e_{2}^{-1}c)\}$
are independent, we can take $a=e_{2}$ and $b=e_{2}^{-1}c$ to have
$ab=c$ with $\{a,\sigma^{i}(b)\}$ is independent for $i=0$ and
$1$. From $\sigma^{2}(b)=\sigma^{2}(e_{2}^{-1})\sigma^{2}(c)=-e_{2}^{-1}c=-b$
and $\sigma^{3}(b)=-\sigma(b)$, we easily deduce that $\{a,\sigma^{i}(b)\}$
is independent for all $i\in\mathbb{Z}$.

Thus, we are left with the situation that not both $\{e_{2},e_{2}^{-1}c\}$
and $\{e_{2},\sigma(e_{2}^{-1}c)\}$ are independent. We cannot have
$\sigma(e_{2}^{-1}c)=\alpha e_{2}$ for any $\alpha\in k_{0}$. Otherwise,
$\sigma(c)=\alpha e_{2}\sigma(e_{2})$, and so $\sigma^{2}(c)=-\alpha\sigma(e_{2})e_{2}=-\alpha(c)$,
showing that $\alpha(c)$ is an element of $k_{1}$. But then $c\in\sigma^{-1}(k_{1})=k_{1}$,
a contradiction. Thus, $\{e_{2},\sigma(e_{2}^{-1}c)\}$ is independent.
Consequently, we have to have $e_{2}^{-1}c=\beta e_{2}$ for some
$\beta\in k_{0}$. So, $c=\beta e_{2}^{2}$.

To complete the proof, we shall take $a=\sigma(e_{2})$, $b=a^{-1}c=\beta\sigma(e_{2}^{-1})e_{2}^{2}$,
and show that for every $i\in\mathbb{Z}$, $\{a,\sigma^{i}(b)\}$
is independent. Then from Lemma \ref{lem:l.indep.} we would get $aL+\sigma^{i}(b)L=k$.

Let us show that $a$ and $b$ are independent. Suppose that this
is not the case, and let $a=\gamma b$ for some $\gamma\in k_{0}$.
Then $\sigma(e_{2}^{2})=\sigma(e_{2})^{2}=a^{2}=a\cdot\gamma b=\gamma c=\gamma\beta e_{2}^{2}$,
and so
\[
e_{2}^{2}=(-e_{2})^{2}=(\sigma^{2}(e_{2}))^{2}=\sigma^{2}(e_{2}^{2})=\sigma(\sigma(e_{2}^{2}))=\gamma\beta\sigma(e_{2}^{2})=(\gamma\beta)^{2}e_{2}^{2}.
\]
Therefore, $(\gamma\beta)^{2}=1$. If $\gamma\beta=-1$, then $\sigma(e_{2}^{2})=-e_{2}^{2}$,
i.e., $e_{2}^{2}\in k_{1}$. But then $c=\beta e_{2}^{2}\in k_{1}$,
which is not the case. If $\gamma\beta=1$, then $\sigma(e_{2}^{2})=e_{2}^{2}$,
and so $(\sigma(e_{2})-e_{2})(\sigma(e_{2})+e_{2})=0$. Since $\{y,\sigma(y),\sigma^{2}(y),\sigma^{3}(y)\}$
is a basis of $k$, we have $\sigma(e_{2})-e_{2}=\sigma(y)-\sigma^{3}(y)-y+\sigma^{2}(y)\not=0$
and $\sigma(e_{2})+e_{2}=\sigma(y)-\sigma^{3}(y)+y-\sigma^{2}(y)\not=0$.
Thus, $\gamma\beta=1$ leads to a contradiction. Therefore, $a$ and
$b$ are indeed independent.

Next, assume that $a$ and $\sigma(b)$ are dependent. Thus, $\sigma(b)=\zeta a=\zeta\sigma(e_{2})$
for some $\zeta\in k_{0}$. Then $\zeta\sigma(e_{2})=\sigma(b)=\sigma(\beta\sigma(e_{2}^{-1})e_{2}^{2})=\beta\sigma^{2}(e_{2}^{-1})\sigma(e_{2}^{2})=-\beta e_{2}^{-1}\sigma(e_{2})^{2}$,
this implies that $e_{2}^{-1}\sigma(e_{2})=-\beta^{-1}\zeta\in k_{0}$.
Thus,
\[
e_{2}^{-1}\sigma(e_{2})=\sigma(e_{2}^{-1}\sigma(e_{2}))=\sigma(e_{2}^{-1})\sigma^{2}(e_{2})=-\sigma(e_{2})^{-1}e_{2},
\]
and so $\sigma(e_{2}^{2})=-e_{2}^{2}$. But then $e_{2}^{2}\in k_{1}$,
yielding $c=\beta e_{2}^{2}\in k_{1}$, a contradiction. Therefore,
$a$ and $\sigma(b)$ have to be independent.

Now, $\sigma^{2}(b)=\sigma^{2}(\sigma(e_{2}^{-1})c)=\sigma(\sigma^{2}(e_{2}^{-1}))\sigma^{2}(c)=-\sigma(e_{2}^{-1})c=-b$,
$\sigma^{3}(b)=\sigma(\sigma^{2}(b))=-\sigma(b)$, $\sigma^{-3}(b)=\sigma^{4-3}(b)=\sigma(b)$,
etc. Therefore, $\{a,\sigma^{2i}(b)\}=\{a,\pm b\}$ and $\{a,\sigma^{2i+1}(b)\}=\{a,\pm\sigma(b)\}$
for all $i\in\mathbb{Z}$. This shows that $\{a,\sigma^{i}(b)\}$
is independent for every $i\in\mathbb{Z}$. The proof is now complete.
\end{proof}
\begin{lem}
\label{lem:special-a_s}Let $f=\sum_{i=s}^{\infty}a_{s}x^{s}\in D$
with $a_{s}\not\in k_{1}$. Then $f=f_{1}f_{2}$ for some $f_{1}=\sum_{i=s}^{\infty}b_{i}x^{i}$
and $f_{2}=\sum_{j=0}^{\infty}c_{j}x^{j}$ in $D$ with $b_{i},c_{j}\in L$
for all $i$ and $j$. Consequently, $f$ is a product of two commutators
in $D$.
\end{lem}

\begin{proof}
We have either $\sigma^{2}(a_{s})\not=a_{s}$ or $\sigma^{2}(a_{s})=a_{s}$.
Since $a_{s}\not\in k_{1}$, Lemma~\ref{lem:special_conditions}
applies, and so there are $b_{s},c_{0}'\in L$ such that $a_{s}=b_{s}c_{0}'$
and $b_{s}L+\sigma^{\ell}(c_{0}')L=k$ for all $\ell\in\mathbb{Z}$.
Let $c_{0}=\sigma^{-s}(c_{0}')$. Then $a_{s}=b_{s}\sigma^{s}(c_{0})$
and $b_{s}L+\sigma^{\ell}(c_{0})L=k$ for all $\ell\in\mathbb{Z}$.
Next, we chose $b_{s+1},c_{1}'\in L$ such that $b_{s}c_{1}'+\sigma^{s+1}(c_{0})b_{s+1}=a_{s+1}$.
By taking $c_{1}=\sigma^{-s}(c_{1}')$, we then have $b_{s}\sigma^{s}(c_{1})+b_{s+1}\sigma^{s+1}(c_{0})=a_{s+1}$.
Assume that $t\geq1$ and we have chosen $b_{s+t},c_{t}'\in L$ such
that
\[
b_{s}c_{t}'+\sigma^{s+t}(c_{0})b_{s+t}=a_{s+t}-\sum_{r=1}^{t-1}b_{s+r}\sigma^{s+r}(c_{t-r}),
\]
and set $c_{t}=\sigma^{-s}(c_{t}')$. As $b_{s}L+\sigma^{s+t+1}(c_{0})L=k$,
we can find $b_{s+t+1}$ and $c_{t+1}'$ from $L$ such that
\[
b_{s}c_{t+1}'+b_{s+t+1}\sigma^{s+t+1}(c_{0})=a_{s+t+1}-\sum_{r=1}^{t}b_{s+r}\sigma^{s+r}(c_{t+1-r}).
\]
Set $c_{t+1}=\sigma^{-s}(c_{t+1}')$. This process can be continued
and we eventually obtain $f_{1}$ and $f_{2}$ as required.

As $f_{1},f_{2}\in\{\sum_{i=s}^{\infty}a_{i}x^{i}\mid a_{i}\in L\}=[x,D]$,
the last statement is clear.
\end{proof}
We are ready to prove the order $4$ case.
\begin{proof}[Proof of the case when $\sigma$ is of order $4$]
Let $f=\sum_{i=s}^{\infty}a_{i}x^{i}\in D$ with $a_{s}\neq0$. Write
$s=4q+r$, where $q,r\in\mathbb{Z}$ and $0\leq r<4$. For $r\in\{0,1,3\}$,
we set $u,v\in\mathbb{Z}$ as the following:
\[
\begin{cases}
u=s+1\text{ and }v=-1, & \text{if }r=0\text{ or }1;\\
u=s-1\text{ and }v=1, & \text{if }r=3.
\end{cases}
\]
Then $s=u+v$, $4\nmid(u-v)$, $4\nmid u$ and $4\nmid v$, and $f$
can be written a product of two commutators by Lemma \ref{lem:i+j,i=000020neq=000020j}.

Suppose that $s=4q+2$. In view of Lemma~\ref{lem:special-a_s},
we only need to care about the situation when $a_{s}\in k_{1}$. Recall
that we have a basis $\{y,\sigma(y),\sigma^{2}(y),\sigma^{3}(y)\}$
for $k$. Consider $g=yfy^{-1}=\sum_{i=s}^{\infty}b_{s}x^{s}$, where
$b_{s}=ya_{s}\sigma^{s}(y^{-1})=a_{s}y\sigma^{2}(y^{-1})$. We claim
that $b_{s}\not\in k_{1}$. On the contrary, assume $\sigma(b_{s})=-b_{s}$.
With $\sigma(a_{s})=-a_{s}$, we get $-a_{s}y\sigma^{2}(y^{-1})=-b=\sigma(b)=\sigma(a_{s}y\sigma^{2}(y^{-1}))=-a_{s}\sigma(y\sigma^{2}(y^{-1}))$.
Then $\sigma(y\sigma^{2}(y^{-1}))=y\sigma^{2}(y^{-1})$, which says
that $y\sigma^{2}(y^{-1})\in k_{0}$. Let $\alpha=y\sigma^{2}(y^{-1})=y\sigma^{2}(y)^{-1}$.
Then $y-\alpha\sigma^{2}(y)=0$, contradicting the fact that $\{y,\sigma(y),\sigma^{2}(y),\sigma^{3}(y)\}$
is a basis of $k$. Therefore, $b_{s}$ cannot be in $k_{1}$. By
Lemma~\ref{lem:special-a_s}, $yfy^{-1}=[x,f_{1}][x,f_{2}]$ for
some $f_{1},f_{2}\in D$. Therefore, we have
\[
f=y^{-1}[x,f_{1}][x,f_{2}]y=[y^{-1}xy,y^{-1}f_{1}y][y^{-1}xy,y^{-1}f_{2}y].
\]
This finishes the proof for $\sigma$ being of order $4$, and completes
the proof of the Theorem.
\end{proof}

\section*{Acknowledgment}

The first author would like to express his sincere gratitude to Professor
Matej Brešar for his invaluable support during his visit to the University
of Ljubljana and the University of Maribor. The original question,
posed in the paper by Gardella and Thiel, was highlighted by Professor
Brešar.

\end{document}